\documentclass[11pt]{article}

\usepackage{amsmath,amsfonts,amssymb,amsthm,graphics,amscd}

\newtheorem{theorem}{Theorem}[section]
\newtheorem{lemma}{Lemma}[section]
\newtheorem{corollary}{Corollary}[section]

\theoremstyle{definition}

\def\A{{\cal A}}
\def\sbmatrix{\left[\begin{array}}
\def\endsbmatrix{\end{array}\right]}

\begin{document}

\title{On the generalized Drazin inverse of
the sum in a Banach algebra}

\author{Dijana Mosi\' c\footnote{The author is supported by the Ministry of Education and Science,
Republic of Serbia, grant no. 174007.},
Daochang Zhang\footnote{The corresponding author is supported by NSFC (No. 11371165; No. 11401249),
the Scientific and Technological Development Program Foundation of Jilin
Province, China (No. 20170520052JH; No. 20130522097JH),
the Doctor Research Start-up Fund of Northeast Electric Power University (No. BSJXM-201525).
}}

\date{}

\maketitle

\begin{abstract}
The objective of this paper is to study the existence of the
generalized Drazin inverse of the sum $a+b$ in a Banach algebra
and present explicit expressions for the generalized Drazin
inverse of this sum, under new conditions.

\medskip

{\it Key words and phrases\/}: generalized Drazin inverse,
additive properties, Banach algebras.

2010 {\it Mathematics subject classification\/}: 46H05, 47A05,
15A09.
\end{abstract}


\section{Introduction}

Throughout this paper, $\A$ will denote a complex unital Banach
algebra with unit 1. We will use $\A^{-1}$, $\A^{nil}$ and
$\A^{qnil}$ to denote the sets of all invertible, nilpotent and
quasinilpotent elements of $\A$, respectively.

The generalized Drazin inverse of $a\in\A$ (or Koliha--Drazin
inverse of $a$ \cite{K3}) is the unique element $a^d\in\A$ which
satisfies
$$a^daa^d=a^d,\qquad aa^d=a^da,\qquad a-a^2a^d\in\A^{qnil}.$$
The set of all generalized Drazin invertible elements of $\A$ will
be denoted by $\A^d$. For $a\in\A^d$, $a^\pi=1-aa^d$ is the
spectral idempotent of $a$ corresponding to the set $\{0\}$. The
Drazin inverse is a special case of the generalized Drazin inverse
for which $a-a^2a^d\in\A^{nil}$.

If $p\in\A$ is an idempotent, we represent any element $a\in\A$ as
$$a=\sbmatrix{cc} a_{11}&a_{12}\\a_{21}&a_{22}\endsbmatrix,$$
where $a_{11}=pap$, $a_{12}=pa(1-p)$, $a_{21} = (1-p)ap$, $a_{22}=
(1-p)a(1-p)$.

Let $a\in\A^d$. Then we can write
$$a=\sbmatrix{cc}
a_{1}&0\\0&a_{2}\endsbmatrix$$ relative to $p=aa^d$, where
$a_{1}\in(p\A p)^{-1}$ and $a_{2}\in((1-p)\A(1-p))^{qnil}$. The
generalized Drazin inverse of $a$ is given by
$$a^d=\sbmatrix{cc} a^d&0\\0&0\endsbmatrix=\sbmatrix{cc}
a_1^{-1}&0\\0&0\endsbmatrix.$$

Now, we state very useful result on the generalized Drazin inverse
of a triangular block matrix.

\begin{lemma}{\rm \cite[Theorem 2.3]{CGK}} \label{tetriangD}
Let $x=\sbmatrix{cc} a&0\\c&b\endsbmatrix\in\A$ relative to the
idempotent $p\in\A$ and let $y=\sbmatrix{cc}
b&c\\0&a\endsbmatrix\in\A$ relative to the idempotent $1-p$.

{\rm (i)} If $a\in(p\A p)^d$ and $b\in((1-p)\A (1-p))^d$, then
$x,y\in\A^d$ and
$$x^d=\sbmatrix{cc} a^d&0\\u&b^d\endsbmatrix, \qquad y^d=\sbmatrix{cc}
b^d&u\\0&a^d\endsbmatrix,$$ where
$$u=\sum\limits_{n=0}^{\infty}(b^d)^{n+2}ca^na^\pi+
\sum\limits_{n=0}^{\infty}b^\pi b^{n}c(a^d)^{n+2}-b^dca^d.$$

{\rm (ii)} If $x\in\A^d$ and $a\in(p\A p)^d$, then $b\in((1-p)\A
(1-p))^d$ and $x^d$ is given as in part {\rm (i)}.
\end{lemma}

For $a,b\in\A^d$, it is possible that the generalized Drazin
inverse of $a+b$ does not exist or, if $(a+b)^d$ exists, then we
do not always know how to calculate $(a+b)^d$ in terms of $a$,
$b$, $a^d$, $b^d$. Many authors have investigated some special
cases of this problem \cite{CGK,DjW}, but we present some of them
which will be used later.

\begin{lemma}{\rm \cite{MDbmmss-add}}\label{te-add-gD1} If $a\in\A^{qnil}$, $b\in\A^d$ and $ab=b^\pi bab^\pi$, then  $a+b\in\A^d$ and
$$(a+b)^d=\sum_{n=0}^{\infty}(b^d)^{n+1}a^n.$$
\end{lemma}

\begin{lemma}{\rm \cite{DWadd}}\label{le-add-DW} If $a, b\in\A^d$ and $ab=ba$, then  $a+b\in\A^d$
if and only if $1+a^db\in\A^d$. In this case, we have
$$(a+b)^d=a^d(1+a^db)^dbb^d+b^\pi\sum_{n=0}^{\infty}(-b)^{n}(a^d)^{n+1}+\sum_{n=0}^{\infty}(b^d)^{n+1}(-a)^na^\pi.$$
\end{lemma}

Using the assumptions $a^\pi b=b$ and $aba^\pi=0$, the
representation for $(a+b)^d$ was presented in \cite{CGK}. In
\cite{LQB}, the formula for $(a+b)^d$ was given under conditions
which involve $aba^\pi=0$.


In this paper, we give explicit expressions for the generalized
Drazin inverse of $a+b$ in the cases that $aba^\pi=a^\pi b^\pi
bab^\pi a^\pi$ or $aba^\pi=a^\pi baa^\pi$. So, under new
conditions in the literature, the paper studies additive
properties of the generalized Drazin inverse in a Banach algebra.

\section{Additive results}

In the first theorem of this section, for $a,b\in\A^d$ such that
$aba^\pi=a^\pi b^\pi bab^\pi a^\pi$, we investigate the existence
of the generalized Drazin inverse of the sum $a+b$ and give
explicit formula for $(a+b)^d$.

\begin{theorem}\label{te-add-gD11} Let $a,b\in\A^d$ and $aba^\pi=a^\pi b^\pi bab^\pi a^\pi$.
If $a^\pi ba^\pi$ $($or $a^\pi b$ or $ba^\pi$ or $aa^dbaa^d$ or
$aa^db$ or $baa^d$$)$ is generalized Drazin invertible, then
$$a+b\in\A^d \Leftrightarrow c=aa^d(a+b)\in\A^d \Leftrightarrow
(a+b)aa^d\in\A^d \Leftrightarrow aa^d(a+b)aa^d\in\A^d.$$ In this
case,
\begin{eqnarray}(a+b)^d&=&c^d+\sum\limits_{n=0}^{\infty}(b^d)^{n+1}a^{n}a^\pi
-\sum\limits_{n=0}^{\infty}(b^d)^{n+1}a^{n}a^\pi
bc^d\nonumber\\&+&\sum\limits_{n=0}^{\infty}\sum\limits_{k=0}^{\infty}(b^d)^{n+k+2}a^{k}a^\pi
baa^dc^nc^\pi+\sum\limits_{n=0}^{\infty}b^\pi(a+b)^na^\pi
b(c^d)^{n+2}\nonumber\\
&-&\sum_{n=0}^{\infty}\sum_{k=0}^{\infty}(b^d)^{k+1}a^{k+1}(a+b)^{n}a^\pi
b(c^d)^{n+2}.\label{jed-add-gD11}
\end{eqnarray}
\end{theorem}

\begin{proof} We have
the following matrix representations of $a$ and $b$ relative to
$p=aa^d$:
\begin{equation}a=\sbmatrix{cc}
a_{1}&0\\0&a_{2}\endsbmatrix, \quad b=\sbmatrix{cc}
b_{1}&b_{2}\\b_{3}&b_{4}\endsbmatrix,\label{jed-add-gD51}\end{equation} where $a_{1}\in(p\A
p)^{-1}$ and $a_{2}\in((1-p)\A(1-p))^{qnil}$.

By the assumption $aba^\pi=a^\pi b^\pi bab^\pi a^\pi$, we get
$a_1b_2=0$. Since $a_1$ is invertible, we deduce that $b_2=0$. So,
$$b=\sbmatrix{cc}
b_1&0\\b_{3}&b_{4}\endsbmatrix.$$ Using Lemma \ref{tetriangD}, if
one of the elements $a^\pi ba^\pi$ or $a^\pi b$ or $ba^\pi$ is
generalized Drazin invertible, we conclude that
$b_{4}\in((1-p)\A(1-p))^{d}$. Similarly, if one of the elements
$aa^dbaa^d$ or $aa^db$ or $baa^d$ is generalized Drazin
invertible, then $b_{1}\in(p\A p)^{d}$. Applying again Lemma
\ref{tetriangD}, because $b\in\A^d$ and one of the above mentioned
elements is generalized Drazin invertible, $b_{1}\in(p\A p)^{d}$,
$b_{4}\in((1-p)\A(1-p))^{d}$,
\begin{equation}b^d=\sbmatrix{cc}
b_{1}^d&0\\s&b_{4}^d\endsbmatrix \quad {\rm and} \quad
b^\pi=\sbmatrix{cc}
b_1^\pi&0\\-(b_{3}b_{1}^d+b_{4}s)&b_{4}^\pi\endsbmatrix,\label{jed-add-gD71}\end{equation} where
$$s=\sum\limits_{n=0}^{\infty}(b_4^d)^{n+2}b_3b_1^nb_1^\pi+
\sum\limits_{n=0}^{\infty}b_4^\pi
b_4^{n}b_3(b_1^d)^{n+2}-b_4^db_3b_1^d.$$

From
$$\sbmatrix{cc}
0&0\\0&a_{2}b_4\endsbmatrix=aba^\pi=a^\pi b^\pi bab^\pi
a^\pi=\sbmatrix{cc} 0&0\\0&b_{4}^\pi
b_{4}a_2b_{4}^\pi\endsbmatrix,$$ we obtain $a_{2}b_4=b_{4}^\pi
b_{4}a_2b_{4}^\pi$. By Lemma \ref{te-add-gD1}, we observe that
$a_2+b_4\in((1-p)\A(1-p))^{d}$ and
$$(a_2+b_4)^d=\sum_{n=0}^{\infty}(b_4^d)^{n+1}a_2^n.$$
Applying Lemma \ref{tetriangD}, $a+b$ is generalized Drazin
invertible if and only if $c(=aa^d(a+b)=aa^d(a+b)aa^d)=a_1+b_1$ is
generalized Drazin invertible if and only if $(a+b)aa^d$ is
generalized Drazin invertible. In this case,
\begin{equation}(a+b)^d=\sbmatrix{cc}
a_1+b_1&0\\b_{3}&a_2+b_{4}\endsbmatrix^d=\sbmatrix{cc}
c^d&0\\u&(a_2+b_{4})^d\endsbmatrix
,\label{jed-add-gD31}\end{equation}
where
$$u=\sum\limits_{n=0}^{\infty}[(a_2+b_4)^d]^{n+2}b_3c^nc^\pi
+\sum_{n=0}^{\infty}(a_2+b_{4})^\pi(a_2+b_{4})^nb_3(c^d)^{n+2}-(a_2+b_{4})^db_3c^d.$$

The equality $a_{2}b_4=b_{4}^\pi b_{4}a_2b_{4}^\pi$ implies
$a_{2}b_4^d=0$ and
\begin{eqnarray*}
(a_2+b_4)^\pi&=&(1-p)-(a_2+b_4)(a_2+b_4)^d=(1-p)-b_4\sum_{k=0}^{\infty}(b_4^d)^{k+1}a_2^k\\
&=&b_4^\pi-\sum_{k=0}^{\infty}(b_4^d)^{k+1}a_2^{k+1}.
\end{eqnarray*}
Hence,
\begin{eqnarray}
\sum_{n=0}^{\infty}(a_2+b_{4})^\pi(a_2+b_{4})^nb_3(c^d)^{n+2}&=&\sum_{n=0}^{\infty}b_4^\pi(a_2+b_{4})^nb_3(c^d)^{n+2}\nonumber\\
&-&\sum_{n=0}^{\infty}\sum_{k=0}^{\infty}(b_4^d)^{k+1}a_2^{k+1}(a_2+b_{4})^nb_3(c^d)^{n+2}.\nonumber\\\label{jed-add-gD41}
\end{eqnarray}
By the equalities
$$X_1=\sum\limits_{n=0}^{\infty}(b^d)^{n+1}a^{n}a^\pi=\sbmatrix{cc}
0&0\\0&\sum\limits_{n=0}^{\infty}(b_4^d)^{n+1}a_2^n\endsbmatrix=\sbmatrix{cc}
0&0\\0&(a_2+b_{4})^d\endsbmatrix,$$

\begin{eqnarray*}X_2&=&\sum\limits_{n=0}^{\infty}(b^d)^{n+1}a^{n}a^\pi
bc^d=\sbmatrix{cc} 0&0\\0&(a_2+b_{4})^d\endsbmatrix\sbmatrix{cc}
b_1c^d&0\\b_3c^d&0\endsbmatrix\\&=&\sbmatrix{cc}
0&0\\(a_2+b_{4})^db_3c^d&0\endsbmatrix,\end{eqnarray*}

\begin{eqnarray*}X_3&=&\sum\limits_{n=0}^{\infty}\sum\limits_{k=0}^{\infty}(b^d)^{n+k+2}a^{k}a^\pi
baa^dc^nc^\pi=\sum\limits_{n=0}^{\infty}\sum\limits_{k=0}^{\infty}(b^d)^{n+k+2}a^{k}a^\pi
(a^\pi baa^dc^nc^\pi)\\&=&\sum\limits_{n=0}^{\infty}\sbmatrix{cc}
0&0\\0&[(a_2+b_{4})^d]^{n+2}\endsbmatrix\sbmatrix{cc}
0&0\\b_3c^nc^\pi&0\endsbmatrix\\&=&\sbmatrix{cc}
0&0\\\sum\limits_{n=0}^{\infty}[(a_2+b_{4})^d]^{n+2}b_3c^nc^\pi&0\endsbmatrix,
\end{eqnarray*}

$$X_4=\sum\limits_{n=0}^{\infty}b^\pi(a+b)^na^\pi
b(c^d)^{n+2}=\sbmatrix{cc}
0&0\\\sum\limits_{n=0}^{\infty}b_4^\pi(a_2+b_{4})^nb_3(c^d)^{n+2}&0\endsbmatrix,$$

\begin{eqnarray*}X_5&=&\sum_{n=0}^{\infty}\sum_{k=0}^{\infty}(b^d)^{k+1}a^{k+1}(a+b)^{n}a^\pi
b(c^d)^{n+2}\\
&=&\sbmatrix{cc}
0&0\\\sum\limits_{n=0}^{\infty}\sum\limits_{k=0}^{\infty}(b_4^d)^{k+1}a_2^{k+1}(a_2+b_{4})^nb_3(c^d)^{n+2}&0\endsbmatrix,
\end{eqnarray*}
(\ref{jed-add-gD41}) and (\ref{jed-add-gD31}), we obtain that the
formula (\ref{jed-add-gD11}) holds.
\end{proof}

If we suppose that $a\in\A^{qnil}$ in Theorem \ref{te-add-gD11},
we get Lemma \ref{te-add-gD1} as a consequence.

By the following examples, we observe that the conditions
$aba^\pi=0$ (which appears in \cite{LQB}) and $aba^\pi=a^\pi b^\pi
bab^\pi a^\pi$ are independent. Precisely, in the first example,
the condition $aba^\pi=0$ holds, but the condition $aba^\pi=a^\pi
b^\pi bab^\pi a^\pi$ is not satisfied.

{\bf Example 2.1.} We consider the following matrices $a$ and $b$
in the algebra of all complex $3\times 3$ matrices $\A$:
$$a=\sbmatrix{ccc}
0&0&0\\0&0&0\\0&1&0\endsbmatrix, \quad
b=\sbmatrix{ccc}0&0&1\\0&0&0\\0&0&0\endsbmatrix.$$ Since
$a^2=b^2=0$, then $a^\pi=b^\pi=1$. So, $aba^\pi=0$ and $$a^\pi
b^\pi bab^\pi a^\pi=ba=\sbmatrix{ccc}
0&1&0\\0&0&0\\0&0&0\endsbmatrix\neq aba^\pi.$$

Also, we present complex matrices $a$ and $b$ such that
$aba^\pi=a^\pi b^\pi bab^\pi a^\pi$, but $aba^\pi\neq 0$ in the
next example.

{\bf Example 2.2.} Let $\A$ be defined as in Example 2.1 and let
$a,b\in\A$ such that
$$a=b=\sbmatrix{ccc}
0&0&0\\1&0&0\\0&1&0\endsbmatrix.$$ Notice that
$$a^2=\sbmatrix{ccc} 0&0&0\\0&0&0\\1&0&0\endsbmatrix,$$ $a^3=0$
and $a^\pi=1$. Thus, we have $0\neq aba^\pi=a^2=a^\pi a^\pi
a^2a^\pi a^\pi=a^\pi b^\pi bab^\pi a^\pi$.

Now, we give one consequence of Theorem \ref{te-add-gD11}. Notice
that the condition $a^dab=0$ is equivalent to $a^\pi b=b$.

\begin{corollary}\label{cor-add-gD21} Let $a,b\in\A^d$. If $aba^\pi=b^\pi bab^\pi a^\pi$ and $a^dab=0$, then  $a+b\in\A^d$ and
\begin{eqnarray*}(a+b)^d&=&a^d+\sum_{n=0}^{\infty}(b^d)^{n+1}a^{n}a^\pi
+\sum_{n=0}^{\infty}b^\pi(a+b)^{n}b(a^d)^{n+2}\\
&-&\sum_{n=0}^{\infty}\sum_{k=0}^{\infty}(b^d)^{k+1}a^{k+1}(a+b)^nb(a^d)^{n+2}
-\sum_{n=1}^{\infty}(b^d)^{n+1}a^{n}ba^d.
\end{eqnarray*}
\end{corollary}

\begin{proof} The assumptions $aba^\pi=b^\pi bab^\pi a^\pi$ and $a^dab=0$ imply $aba^\pi=a^\pi b^\pi bab^\pi
a^\pi$. Applying Theorem \ref{te-add-gD11}, for $c=a^2a^d$, we
obtain this result.
\end{proof}

If we suppose that $aa^db$ commutes with $a$ in Theorem
\ref{te-add-gD11}, we get the explicit expression for $c^d$ in
terms of $a$, $b$, $a^d$ and $b^d$.

\begin{theorem}\label{te-add-gD31} Let $a,b\in\A^d$, $aba^\pi=a^\pi b^\pi bab^\pi a^\pi$ and $a^2a^db=aa^dba$.
If $a^\pi ba^\pi$ $($or $a^\pi b$ or $ba^\pi$ or $aa^dbaa^d$ or
$aa^db$ or $baa^d$$)$ is generalized Drazin invertible, then
\begin{eqnarray*}a+b\in\A^d &\Leftrightarrow& 1+a^db\in\A^d \Leftrightarrow aa^d(1+a^db)\in\A^d \Leftrightarrow
(1+a^db)aa^d\in\A^d\\& \Leftrightarrow &aa^d(1+a^db)aa^d\in\A^d.
\end{eqnarray*} In this case, $(a+b)^d$ is represented as in {\rm
(\ref{jed-add-gD11})}, where $c=aa^d(a+b)$ and
\begin{eqnarray*}c^d&=&a^d(1+a^db)^dbb^d+aa^db^\pi\sum\limits_{n=0}^{\infty}(-b)^n(a^d)^{n+1}.
\end{eqnarray*}
\end{theorem}

\begin{proof} Using the same notation as in the proof of Theorem
\ref{te-add-gD11}, observe that $a+b$ is generalized Drazin
invertible if and only if $c=aa^d(a+b)=a_1+b_1$ is generalized
Drazin invertible. The equality $a^2a^db=aa^dba$ gives
$a_1b_1=b_1a_1$. By Lemma \ref{le-add-DW}, $c=a_1+b_1$ is
generalized Drazin invertible if and only if
$p+a_1^{-1}b_1=aa^d(1+a^db)(=(1+a^db)aa^d=aa^d(1+a^db)aa^d)$ is
generalized Drazin invertible which is equivalent to $1+a^db$ is
generalized Drazin invertible. Applying Lemma \ref{le-add-DW}, we
have that
\begin{eqnarray}
c^d&=&(a_1+b_1)^d=a_1^{-1}(1+a_1^{-1}b_1)^db_1b_1^d+b_1^\pi\sum_{n=0}^{\infty}(-b_1)^{n}a_1^{-(n+1)}\nonumber\\
&=&a^d(1+a^db)^dbb^d+aa^db^\pi\sum\limits_{n=0}^{\infty}(-b)^n(a^d)^{n+1}.\label{jed-add-gD21}
\end{eqnarray}
\end{proof}

In the case that $aba^\pi=a^\pi baa^\pi$, we obtain the following
result related to the generalized Drazin inverse of the sum $a+b$.

\begin{theorem}\label{te-add-gD41} Let $a,b\in\A^d$ and $aba^\pi=a^\pi baa^\pi$.
If $a^\pi ba^\pi$ (or $a^\pi b$ or $ba^\pi$ or $aa^dbaa^d$ or
$aa^db$ or $baa^d$) is generalized Drazin invertible, then
$$a+b\in\A^d \Leftrightarrow c=aa^d(a+b)\in\A^d \Leftrightarrow
(a+b)aa^d\in\A^d \Leftrightarrow aa^d(a+b)aa^d\in\A^d.$$ In this
case,
\begin{eqnarray}(a+b)^d&=&c^d+\sum\limits_{n=0}^{\infty}(b^d)^{n+1}(-a)^{n}a^\pi
-\sum\limits_{n=0}^{\infty}(b^d)^{n+1}(-a)^{n}a^\pi
bc^d\nonumber\\&+&\sum\limits_{n=0}^{\infty}\left(\sum\limits_{k=0}^{\infty}(b^d)^{k+1}(-a)^{k}a^\pi\right)^{n+2}
baa^dc^nc^\pi+\sum\limits_{n=0}^{\infty}(a+b)^na^\pi
b(c^d)^{n+2}\nonumber\\
&-&\sum_{n=0}^{\infty}(b^d)^{n+1}(-a)^{n}a^\pi(a+b)^{n+1}a^\pi
b(c^d)^{n+2}.\label{jed-add-gD61}
\end{eqnarray}
\end{theorem}

\begin{proof} Suppose that $a$ and $b$ are represented as in (\ref{jed-add-gD51}).
The hypothesis $aba^\pi=a^\pi baa^\pi$ yields $b_2=0$ and $a_2b_4=b_4a_2$.
Hence, $b^d$ and $b^\pi$ are given by (\ref{jed-add-gD71}).
Since $a_{2}\in((1-p)\A(1-p))^{qnil}$, $a_2^d=0$. Using Lemma \ref{le-add-DW},
we conclude that $a_2+b_4\in((1-p)\A(1-p))^{d}$ and
$$(a_2+b_4)^d=\sum_{n=0}^{\infty}(b_4^d)^{n+1}(-a_2)^n.$$
The rest of this proof follows similarly as in the proof of Theorem \ref{te-add-gD11}.
\end{proof}

If we add some extra conditions in Theorem \ref{te-add-gD41}, we
obtain the next results.

\begin{corollary} Let $a\in\A^{qnil}$ and $b\in\A^d$. If $ab=ba$, then
$a+b\in\A^d$ and
$$(a+b)^d=\sum\limits_{n=0}^{\infty}(b^d)^{n+1}(-a)^{n}.$$
\end{corollary}

\begin{corollary} If $a,b\in\A^d$ and $ba^\pi=0$, then
$$a+b\in\A^d \Leftrightarrow c=aa^d(a+b)\in\A^d \Leftrightarrow
a^2a^d+b\in\A^d.$$ In this case,
\begin{eqnarray*}(a+b)^d&=&c^d
+\sum\limits_{n=0}^{\infty}a^na^\pi b(c^d)^{n+2}\nonumber.
\end{eqnarray*}
\end{corollary}

\begin{corollary} Let $a,b\in\A^d$. If $aba^\pi=baa^\pi$ and $a^dab=0$, then
$a+b\in\A^d$ and
\begin{eqnarray*}(a+b)^d&=&a^d+\sum\limits_{n=0}^{\infty}(b^d)^{n+1}(-a)^{n}a^\pi
-\sum\limits_{n=0}^{\infty}(b^d)^{n+1}(-a)^{n}
ba^d\nonumber\\&+&\sum\limits_{n=0}^{\infty}(a+b)^n
b(a^d)^{n+2}-\sum_{n=0}^{\infty}(b^d)^{n+1}(-a)^{n}a^\pi(a+b)^{n+1}
b(a^d)^{n+2}.
\end{eqnarray*}
\end{corollary}

\begin{corollary} Let $a,b\in\A^d$ and $aba^\pi=a^\pi baa^\pi$ and $a^2a^db=aa^dba$.
If $a^\pi ba^\pi$ $($or $a^\pi b$ or $ba^\pi$ or $aa^dbaa^d$ or
$aa^db$ or $baa^d$$)$ is generalized Drazin invertible, then
\begin{eqnarray*}a+b\in\A^d &\Leftrightarrow& 1+a^db\in\A^d \Leftrightarrow aa^d(1+a^db)\in\A^d \Leftrightarrow
(1+a^db)aa^d\in\A^d\\& \Leftrightarrow&
aa^d(1+a^db)aa^d\in\A^d.\end{eqnarray*} In this case, $(a+b)^d$ is
represented as in {\rm (\ref{jed-add-gD61})}, where $c=aa^d(a+b)$
and $c^d$ is given by {\rm (\ref{jed-add-gD21})}.
\end{corollary}

If we define the reverse multiplication in a Banach algebra $\A$
by $a\circ b=ba$, we obtain a Banach algebra $(\A,\circ)$.
Applying the previous theorems to new algebra $(\A,\circ)$, we can
obtain dual results. For example, see a consequences of Theorem
\ref{te-add-gD11} and Theorem \ref{te-add-gD41} in a new algebra
$(\A,\circ)$.

\begin{corollary} Let $a,b\in\A^d$ and $a^\pi ba=a^\pi b^\pi abb^\pi a^\pi$.
If $a^\pi ba^\pi$ $($or $a^\pi b$ or $ba^\pi$ or $aa^dbaa^d$ or
$aa^db$ or $baa^d$$)$ is generalized Drazin invertible, then
$$a+b\in\A^d \Leftrightarrow aa^d(a+b)\in\A^d \Leftrightarrow
e=(a+b)aa^d\in\A^d \Leftrightarrow aa^d(a+b)aa^d\in\A^d.$$ In this
case,
\begin{eqnarray*}(a+b)^d&=&e^d+\sum\limits_{n=0}^{\infty}a^\pi a^{n}(b^d)^{n+1}
-\sum\limits_{n=0}^{\infty}e^dba^\pi
a^{n}(b^d)^{n+1}\nonumber\\&+&\sum\limits_{n=0}^{\infty}\sum\limits_{k=0}^{\infty}
e^\pi e^naa^dba^{k}a^\pi (b^d)^{n+k+2}+\sum\limits_{n=0}^{\infty}
(e^d)^{n+2}ba^\pi (a+b)^nb^\pi\nonumber\\
&-&\sum_{n=0}^{\infty}\sum_{k=0}^{\infty} (e^d)^{n+2}ba^\pi
(a+b)^{n}a^{k+1}(b^d)^{k+1}.
\end{eqnarray*}
\end{corollary}

\begin{corollary} Let $a,b\in\A^d$ and $a^\pi ba=a^\pi aba^\pi$.
If $a^\pi ba^\pi$ (or $a^\pi b$ or $ba^\pi$ or $aa^dbaa^d$ or
$aa^db$ or $baa^d$) is generalized Drazin invertible, then
$$a+b\in\A^d \Leftrightarrow aa^d(a+b)\in\A^d \Leftrightarrow
e=(a+b)aa^d\in\A^d \Leftrightarrow aa^d(a+b)aa^d\in\A^d.$$ In this
case,
\begin{eqnarray*}(a+b)^d&=&e^d+\sum\limits_{n=0}^{\infty}a^\pi (-a)^{n}(b^d)^{n+1}
-\sum\limits_{n=0}^{\infty}e^dba^\pi (-a)^{n}(b^d)^{n+1}
\nonumber\\&+&\sum\limits_{n=0}^{\infty} e^\pi
e^naa^db\left(\sum\limits_{k=0}^{\infty}a^\pi(-a)^{k}
(b^d)^{k+1}\right)^{n+2} +\sum\limits_{n=0}^{\infty}
(e^d)^{n+2}ba^\pi (a+b)^n\nonumber\\
&-&\sum_{n=0}^{\infty}(e^d)^{n+2}ba^\pi(a+b)^{n+1}a^\pi
(-a)^{n}(b^d)^{n+1}.
\end{eqnarray*}
\end{corollary}

Dijana Mosi\'c

\bigskip

Faculty of Sciences and Mathematics, University of Ni\v s, P.O.
Box 224, 18000 Ni\v s, Serbia

\bigskip

{\it E-mail:} {\tt dijana@pmf.ni.ac.rs}

\bigskip

Daochang Zhang

\bigskip

College of Sciences, Northeast Electric Power University, Jilin 132012, China

\bigskip

{\it E-mail:} {\tt daochangzhang@126.com}

\end{document}